\documentclass[12pt]{article}

\usepackage{amssymb}
\usepackage{amsthm}
\usepackage{amsmath}
\usepackage{verbatim}
\usepackage{mathrsfs}
\usepackage{enumerate}
\usepackage[margin=1in]{geometry}
\usepackage[usenames,dvipsnames] {color}

\usepackage{bm}

\def\mcc{M\raise.5ex\hbox{c}C}
\def\mccarthy{M\raise.5ex\hbox{c}Carthy}

\def\eg{{\it e.g. }}

\def\h{{\cal H}}

\def\N{{\cal N}}

\def\m{Mult}

\def\={\ = \ }

\def\C{\mathbb C}

\def\be{\setcounter{equation}{\value{theorem}} \begin{equation}}
\def\ee{\end{equation} \addtocounter{theorem}{1}}
\def\beq{\begin{eqnarray*}}
\def\eeq{\end{eqnarray*}}

\def\bp{{\sc Proof: }}
\def\ep{{}{\hfill $\Box$} \vskip 5pt \par}

\def\bl{\begin{lemma}}
\def\el{\end{lemma}}
\def\bt{\begin{theorem}}
\def\et{\end{theorem}}
\def\bprop{\begin{prop}}
\def\eprop{\end{prop}}
\def\bd{\begin{definition}}
\def\ed{\end{definition}}
\def\br{\begin{remark}}
\def\er{\end{remark}}
\def\bexer{\begin{exercise}}
\def\eexer{\end{exercise}}
\def\bfig{\begin{figure}}
\def\efig{\end{figure}}

\numberwithin{equation}{section}

\title{Wandering Montel Theorems for Hilbert Space Valued Holomorphic Functions}
\author{Jim Agler
\thanks{Partially supported by National Science Foundation Grant
DMS 1361720}
\\ U.C. San Diego\\ La Jolla, CA 92093
\and
John E. M\raise.5ex\hbox{c}Carthy
\thanks{Partially supported by National Science Foundation Grant  
DMS 1565243
}
\\ Washington University\\ St. Louis, MO 63130
}

\definecolor{dark_purple}{rgb}{0.4, 0.0, 0.4}

\def\d{\mathbb{D}}
\def\c{\mathbb{C}}

\def\bbm{\mathbb{M}}
\def\N{\mathbb{N}}

\def\m{\mathcal{M}}
\def\n{\mathcal{N}}
\def\h{\mathcal{H}}

\def \b{\mathcal{B}}

\def \x{\mathcal{X}}

\def\calc{\mathcal{C}}

\def\calp{\mathcal{P}}

\def\ip#1#2{\langle #1,#2 \rangle}

\def\be{\begin{equation}}
\def\ee{\end{equation}}

\newcommand\de{\delta}
\renewcommand\bd{B_\delta}
\renewcommand\calc{{\mathcal C}}
\newcommand\md{\bbm^d}

\def\set#1#2{\{ #1 \, : \, #2\}}

\def\spn{{\rm span\,}}
\def\norm#1{\| #1 \|}

\def\s0{s_0}
\def\p0{p_0}

\DeclareMathOperator{\hol}{Hol}
\DeclareMathOperator{\her}{Her}

\DeclareMathOperator{\id}{id}
\newcommand{\tensor}[2]{\text{ }{\begin{smallmatrix} #1 \\ \otimes\\ #2\end{smallmatrix}}\text{  }}

\renewcommand\O{\Omega}

\newcommand\holnc{{\rm Hol}^{\rm nc}_\h(\Omega)}

\begin{document}
\maketitle



\bibliographystyle{plain}
\allowdisplaybreaks

\theoremstyle{definition}
\newtheorem{defin}[equation]{Definition}
\newtheorem{lem}[equation]{Lemma}
\newtheorem{cor}[equation]{Corollary}
\newtheorem{prop}[equation]{Proposition}
\newtheorem{thm}[equation]{Theorem}
\newtheorem{claim}[equation]{Claim}
\newtheorem{ques}[equation]{Question}
\newtheorem{prob}[equation]{Problem}
\newtheorem{fact}[equation]{Fact}
\newtheorem{rem}[equation]{Remark}
\newtheorem{rems}[equation]{Remarks}
\newtheorem{notation}[equation]{Notation}
\newtheorem{exam}[equation]{Example}
\newtheorem{con}[equation]{Conjecture}

\section{Introduction}
\label{seca}

\subsection{Commutative Theory}

Let $\Omega$ be an open set in $\c^d$ and assume that $\{u^k\}$ is a sequence in $\hol(\Omega)$, the algebra of of holomorphic functions on $\Omega$ equipped with the topology of uniform convergence on compact subsets.
The classical Montel Theorem asserts that if $\{u^k\}$ is locally uniformly bounded on $\Omega$, then there exists a subsequence $\{u^{k_l}\}$ that converges in $\hol(\Omega)$.

It is well known that if $\x$ is an infinite dimensional Banach space, then Montel's Theorem 
breaks down for 
$\hol_\x(\Omega)$, the space of
$\x$-valued holomorphic functions, see \eg \cite{an00,kk03}.
For example, if $\x=\ell^2$ and $\{f^k\}$ is a locally uniformly bounded sequence of holomorphic functions on $\Omega$, then the sequence
\[
\begin{pmatrix}f^1(\lambda)\\0 \\0\\ \vdots\end{pmatrix},\
\begin{pmatrix}0\\ f^2(\lambda)\\0 \\ \vdots\end{pmatrix},\
\begin{pmatrix}0\\ 0\\ f^3(\lambda)\\ \vdots\end{pmatrix}, \ldots
\]
is a locally uniformly bounded sequence that will have a convergent subsequence only if there exists a subsequence $\{f^{k_l}\}$ that converges uniformly to 0 on $\Omega$.

Observe that the problem in the example given above is that while for all $\lambda\in \Omega$, $u^k$ converges weakly to 0, it needn't be the case that $u^k(\lambda)$ converges in norm for any $\lambda \in \Omega$.
 However, just as in the case of the classical proof of Montel's Theorem that uses the Arzela-Ascoli Theorem, if one assumes that $\{u^k\}$ is well behaved pointwise on a large enough set, then one \emph{can} conclude uniform convergence in norm on compact sets.
For example, consider the following  theorem by Arendt and Nikolski \cite[Cor. 2.3]{an00}
\begin{thm}
\label{thman}
Let $\O$ be an open connected set in $\C$, and let $u^k$ be a sequence in 
$\hol_\x(\Omega)$
 that is locally bounded. Assume that $$\O_0 := \{ z \in \O: \{ u^k(z) : k \in \N \} \ \text{
is\ relatively\ compact\ in\ }X \} $$ has an accumulation point in $\O$. Then there exists a subsequence which converges to a holomorphic function uniformly on compact subsets of $\O$.
\end{thm}


Theorem~\ref{thman} deals with the difficulty by making strong additional assumptions about the point-wise behavior of $\{u^k\}$, assumptions that may not hold in desirable applications. The central idea of this paper, for Hilbert space valued functions, is instead to use a sequence of unitaries to push (most of) the range of the functions into a finite-dimensional space. 
Here is our first main result.

\begin{thm}\label{int.thm.10}
If $\Omega$ is an open set in $\c^d$, $\h$ is a Hilbert space, and $\{u^k\}$ is  a locally uniformly bounded sequence in $\hol_\h(\Omega)$, then there exists a sequence $\{U^k\}$
of unitary operators on $\h$ such that $\{U^k u^{k}\}$ has a subsequence that converges in $\hol_\h(\Omega)$.
\end{thm}

We prove Theorem~\ref{int.thm.10} in Section~\ref{secb}.
In Sections~\ref{secc} and \ref{secd} we consider  versions for non-commutative functions.
These functions have been extensively studied recently---see \eg 
\cite{po06,bgm06,akv06,hkm11b,hm12,kvv14,pas14,bal15,hptv16,bmv16a}. Before stating our results,
we must spend a little time explaining some definitions.

\subsection{Non-commutative theory}

In commutative analysis, one studies holomorphic functions  defined on domains in $\c^d$. In noncommutative analysis one studies holomorphic functions defined on domains in $\bbm^d$, the \emph{$d$-dimensional nc universe}. For each $n$ we let $\bbm_n^d$ denote the set of $d$-tuples of $n\times n$ matrices. We then let
\[
\bbm^d =\bigcup_{n=1}^\infty \bbm_n^d.
 \]
When $E$ is a subset of $\bbm^d$, then for each $n$, we adopt the notation
\[
E_n = E \cap \bbm_n^d.
\]

In noncommutative analysis one studies \emph{graded} functions, i.e., functions $f$ defined on subsets $E$ of $\bbm^d$, that satisfy
 \be\label{hol.10}
 \forall_n\ \ \forall_{\lambda \in E_n}\ \
 f(\lambda) \in \bbm_n.
 \ee

 $\bbm^d$ carries a topology, the so-called \emph{finite topology}\footnote{Subsequently, we shall consider other topologies as well.},
 wherein a set $\Omega$ is deemed to be open precisely when
 \[
 \forall_n\ \ \Omega_n \text{ is open in } \bbm_n^d.
 \]
 With this definition, note that a graded function $f:E \to \bbm^1$ is finitely continuous if and only if $f|E_n$ is continuous for each $n$ and also that a set $K\subseteq \bbm^d$ is finitely compact if and only if there exists $n$ such that $E_m =\varnothing$ when $m>n$ and $E_m$ is compact when $m\le n$.

 If $\Omega$ is finitely open in $\bbm^d$, then for each $n$, $\Omega_n$ can be identified with an open set in $\c^{dn^2}$ in an obvious way.  
 If, in addition, $f$ is a graded function on $\Omega$, then we say that $f$ is \emph{holomorphic on $\Omega$} if for each $n$, $f|\Omega_n$ is a holomorphic mapping 
 of $\Omega_n$ into $\bbm_n$. We let $\hol(\Omega)$ denote the collection of graded holomorphic functions.
 
 It is also possible to consider $\h$-valued holomorphic functions in the noncommutative setting. One particularly concrete way to do this is to realize in the scalar case just considered that \eqref{hol.10} is equivalent to asserting that
\[
 \forall_n\ \ \forall_{\lambda \in E_n}\ \
 f(\lambda) \in \b(\c^n,\c^n).
\]
We therefore replace the former definition (that $f$ be graded) with the requirement that $f$ be a \emph{graded $\h$-valued function}, i.e., that
\[
 \forall_n\ \ \forall_{\lambda \in E_n}\ \
 f(\lambda) \in \b(\c^n,\c^n\otimes \h).
\]
Just as before, we declare a graded $\h$-valued function defined on a finitely open set $\Omega$ in $\bbm^d$ to be holomorphic if for each $n$, $f|\Omega_n$ is a holomorphic 
mapping of $\Omega_n$ into $\b(\c^n,\c^n\otimes \h)$. We let $\hol_\h(\Omega)$ denote the collection of graded $\h$-valued functions and view $\hol_\h(\Omega)$ as a
 complete metric space endowed with the
 topology of uniform convergence on finitely compact subsets of $\Omega$.
 
 A special class of graded functions arise by formalizing certain algebraic properties of free polynomials.
  If $E\subseteq \bbm^d $ we say that $E$ is an \emph{nc-set} if $E$ is closed with respect to direct sums.
  We define the class of \emph{nc-functions} as follows.
 \begin{defin}\label{nc.def.10}
Let $\h$ be a Hilbert space, $E$ an nc-set, and assume that $f$ is a function defined on $E$. We say that $f$ is an \emph{$\h$-valued nc-function on $E$} if the following conditions hold.
 \begin{enumerate}[(i)]
 \item $f$ is \emph{$\h$-graded}, i.e.,
 \[
 \forall_n\ \ \forall_{\lambda \in E\cap \bbm_n}\ \
 f(\lambda) \in \b(\c^n,\c^n \otimes \h).
 \]
 \item $f$ \emph{preserves direct sums}, i.e.,
 \[
\forall_{\lambda,\mu \in E}\ \ \lambda\oplus \mu \in E \implies  f(\lambda \oplus \mu) =f(\lambda) \oplus f(\mu).
 \]
 In this formula, if $\lambda \in E_m$ and $\mu\in E_m$, we identify $\c^m \oplus \c^n$ and $\c^{m+n}$ and identify $(\c^m\otimes \h) \oplus (\c^n \otimes \h)$ and $\c^{m+n} \otimes \h$.
 \item $f$ \emph{preserves similarity}, i.e.,
 \[
 f(S \lambda S^{-1}) =(S\otimes \id_\h) f(\lambda)S^{-1}
 \]
 whenever $n\ge 1$, $S \in \bbm_n$ is invertible, and both $\lambda$ and $S\lambda S^{-1}$ are in $E_n$.
 \end{enumerate}
 \end{defin}

 When $f:E \to \bbm^1 \otimes \h$ is an nc-function and $E$ is a finitely open nc-set then Condition (iii) above becomes very strong and yields
 the following proposition which lies at the heart of nc analysis (see \cite{hkm11b} or \cite[Thm. 7.2]{kvv14}).
 \begin{prop}\label{nc.prop.10}
 Let $\Omega$ be a finitely open nc-set. If $f$ is a bounded nc-function defined on $\Omega$, then $f$ is holomorphic on $\Omega$.
 \end{prop}

 Proposition \ref{nc.prop.10} suggests the following terminology. We say that a set $\Omega \subseteq \bbm^d$ is an \emph{nc-domain} if $\Omega$ is a finitely open nc-set and we say that a topology $\tau$ on $\bbm^d$ is an \emph{nc-topology} if $\tau$ has a basis consisting of nc-domains. We then define special classes of  functions in noncommuting variables as follows.
  \begin{defin}\label{nc.def.20}
 Let $\Omega \in \bbm_n^d$, $\tau$ be an nc-topology, and assume that $f:\Omega \to \bbm^1\otimes \h$ is an $\h$-valued function. We say that $f$ is 
  \emph{$\tau$-holomorphic} if $f$ is a $\tau$-locally bounded nc function on $\Omega$.
  \footnote{i.e., $f$ is an nc-function on $\Omega$ in the sense of Definition \ref{nc.def.10} and for each $\lambda \in \Omega$, there exists $B \subseteq \Omega$ such that $\lambda \in B \in \tau$ and $f|B$ is bounded. }
  We let $\hol_\h^\tau(\Omega)$ denote the collection of $\tau$-holomorphic $\h$-valued functions defined on $\Omega$.
 \end{defin}
 Evidently, Proposition \ref{nc.prop.10} guarantees that if $\tau$ is an nc-topology, and $f$ is a $\tau$-holomorphic function in the sense of Definition \ref{nc.def.20}, then $f$ is holomorphic, i.e.,
 \[
 \hol_\h^\tau (\Omega) \subseteq \holnc \subseteq  \hol_\h(\Omega),
 \]
 where $\holnc$ denotes the set of functions in $\hol_\h(\Omega)$ that are nc.

We can now state our second main result, the non-commutative version of Theorem~\ref{int.thm.10}.
 \begin{thm}\label{nc.thm.101}
Assume that $\tau$ is an nc-topology, $\Omega\in \tau$, $\h$ is a Hilbert space, and $\{u^k\}$ is  a $\tau$-locally uniformly bounded sequence in $\hol_\h^\tau(\Omega)$. There exist $u \in \hol_\h^\tau(\Omega)$,  a sequence $\{U^k\}$
of unitary operators on $\h$,  and an increasing sequence of indices $\{k_l\}$ such that
$  (\id_n \otimes U^{k_l}) \ u^{k_l} \to u$ in $\hol(\Omega)$.
\end{thm}

%
As an application of  Theorem~\ref{nc.thm.101}, in Section~\ref{sece} we
prove that the cones
 \[
 \calp=\set{u(\mu)^*u(\lambda)}{u \in \hol_\h(\Omega) \text{ for some Hilbert space } \h}
 \]
 and
 \[
 \calc \  = \ 
 \{ 
 \tensor{\id_{\c^J}}{u(\mu)^*}  \left(  \id - 
 \tensor{\de(\mu)^* \de (\lambda ) }{\id_\h} \right)
  \tensor{\id_{\c^J}}{u(\lambda)}
  \ : \ u \in \hol_\h(\bd) \ {\rm and\ }u\ {\rm is\ nc}
 \}
 \]
 are closed. In this last formula, $\de$ is a $J$-by-$J$ matrix of free polynomials,  
 and $\bd = \{ x : \| \de(x) \| < 1 \}$ is a non-commutative polynomial polyhedron.
(We adopt the convention of \cite{ptd13} and write the tensors  vertically for legibility.)
 
 Proving that the cones are closed is the key step 
in proving realization
 formulas for free holomorphic functions---see \cite{amfree, amfreeII, bmv16b}.
 
 In Section~\ref{secf} we show that the assumptions in
 Proposition  \ref{hol.prop.10}
 can be weakened to just requiring convergence on a set of uniqueness.

\section{A Montel Theorem for Hilbert Space Valued Holomorphic Functions}
\label{secb}

In this section we prove 
Theorem \ref{int.thm.10} from the introduction.
\subsection{Notation and Definitions}
If $\Omega$ is an open set in $\c^d$, $\h$ is a Hilbert space, we let $\hol_\h(\Omega)$ denote the space of holomorphic $\h$-valued functions on $\Omega$.
If $u\in \hol_\h(\Omega)$ and $E \subseteq\Omega$, we let
 \[
 \norm{u}_E = \sup_{\lambda \in E} \norm{u(\lambda)}_\h.
 \]
If $\norm{u}_\Omega <\infty$ then we say that $u$ is \emph{bounded on $\Omega$}. If $\{u^k\}$ is sequence in $\hol_\h(\Omega)$, we say that $\{u^k\}$
is \emph{uniformly bounded on $\Omega$} if
\[
\sup_k \norm{u^k}_\Omega <\infty,
\]
and we say that $\{u^k\}$ is \emph{locally uniformly bounded on $\Omega$} if for each $\lambda \in \Omega$, there exists a neighborhood $B$ of $\lambda$ such that $\{u^k\}$ is uniformly bounded on $B$. Recall that if such a neighborhood exists, then a Cauchy Estimate implies that $\{u^k\}$ is \emph{equicontinuous at $\lambda$}, i.e., for each $\varepsilon >0$ there exists  a ball $B_0$ such that $\lambda \in B_0 \subseteq B$ and
\[
\forall_{\mu \in B_0}\  \forall_k\ \   \norm{u^k(\mu) -u^k(\lambda)} < \varepsilon.
\]

We equip $\hol_\h(\Omega)$ with the usual topology of \emph{uniform convergence on compacta}. Thus, a sequence $\{u^k\}$ in  $\hol_\h(\Omega)$ is \emph{convergent} precisely when there is a function $u\in \hol_\h(\Omega)$ such that
\[
\lim_{k\to \infty} \norm{u^k-u}_E =0
\]
for every compact $E\subseteq \Omega$. We say that a sequence $\{u^k\}$ in $\hol_\h(\Omega)$ is a \emph{Cauchy sequence} if for each compact $E\subseteq \Omega$, $\{u^k\}$ is \emph{uniformly Cauchy on $E$}, i.e., for each $\varepsilon>0$, there exists $N$ such that
\[
k,l \ge N \implies  \norm{u^k-u^l}_E < \varepsilon.
\]
It is well known that $\hol_\h(\Omega)$ is \emph{complete}, i.e.,  every Cauchy sequence in $\hol_\h(\Omega)$ converges. The following result is proved in \cite[Thm. 2.1]{an00}; we include a proof that easily generalizes to
Proposition~\ref{hol.prop.10}.

\begin{prop}\label{int.prop.10}
Assume that $\Omega$ is an open set in $\c^d$, $\{\lambda_i\}$ is a dense sequence in $\Omega$, and $\h$ is a Hilbert space. If $\{u^k\}$ is sequence in $\hol_\h(\Omega)$ that is locally uniformly bounded on $\Omega$, and for each fixed $i$, $\{u^{k}(\lambda_i)\}$ is a convergent sequence in $\h$, then $\{ u^{k}\}$ is a convergent sequence in $\hol_\h(\Omega)$.
\end{prop}
\bp
Fix a compact set $E\subseteq \Omega$ and $\varepsilon>0$. Note that as $\{u^k\}$  is assumed to be locally uniformly bounded on $\Omega$,  $\{u^k\}$ is equicontinuous at each point of $\Omega$. Hence, as $E$ is compact, we may construct a finite collection $\set{B_r}{r=1,\ldots,m}$ of open balls in $\c^d$ such that
\be\label{com.50}
E \subseteq \bigcup_{r=1}^m B_r \subseteq \Omega
\ee
and
\be\label{com.60}
\forall_r\ \ \forall_{\mu_1,\mu_2 \in B_r}\ \ \forall_k\ \  \norm{u^k(\mu_1)-u^k(\mu_2)}<\varepsilon/3
\ee
As $\{\lambda_i\}$ is assumed dense in $\Omega$,
\be\label{com.70}
\forall_r\ \ \exists_{i_r}\ \  \lambda_{i_r} \in B_r.
\ee
Consequently, as for each fixed $i$, we assume that $\{u^{k}(\lambda_i)\}$ is a convergent (and hence, Cauchy) sequence in $\h$, there exists $N$ such that
\be\label{com.80}
\forall_r\ \   k,j\ge N \implies \norm{u^k(\lambda_{i_r})-u^j(\lambda_{i_r})}<\varepsilon/3.
\ee

Now, fix $\lambda\in E$. Use \eqref{com.50} to choose $r$ such that $\lambda \in B_r$. Use \eqref{com.70} to choose $i_r$ such that $\lambda_{i_r} \in B_r$. As $\lambda$ and $\lambda_{i_r}$ are both in $B_r$, we see from \eqref{com.60} that
\[
\forall_k\ \  \norm{u^k(\lambda)-u^k(\lambda_{i_r})}<\varepsilon/3.
\]
Hence, using \eqref{com.80}, we have that if $k,j \ge N$, then
\begin{align*}
&\norm{u^k(\lambda)-u^j(\lambda)}\\ \\
\le\ \ \ &\norm{u^k(\lambda)-u^k(\lambda_{i_r})}+
\norm{u^k(\lambda_{i_r})-u^j(\lambda_{i_r})}+
\norm{u^j(\lambda_{i_r})-u^j(\lambda)}\\ \\
< \varepsilon.
\end{align*}

Since the concluding estimate in the previous paragraph holds for  an arbitrary point $\lambda\in E$,  $\{u^k\}$ is uniformly Cauchy on $E$. Since $E$ is an arbitrary compact subset of $\Omega$, $\{u^k\}$ is a Cauchy sequence in $\hol_\h(\Omega)$. Therefore, $\{u^k\}$ converges in $\hol_\h(\Omega)$.
\subsection{The Proof of Theorem \ref{int.thm.10}}
Theorem \ref{int.thm.10} follows quickly from Proposition \ref{int.prop.10} and the following lemma.
\begin{lem}\label{com.lem.10} {\bf Wandering Isometry Lemma.}
Assume that $\Omega$ is an open set in $\c^d$, $\{\lambda_i\}$ is a sequence in $\Omega$, and $\h$ is 
a Hilbert space. If $\{u^k\}$ is sequence in $\hol_\h(\Omega)$ that is locally uniformly bounded on $\Omega$, then there exists a subsequence $\{u^{k_l}\}$ and a sequence $\{V^l\}$
of unitary operators on $\h$ such that for each fixed $i$, $\{V^l u^{k_l}(\lambda_i)\}$ is a convergent sequence in $\h$.
\end{lem}
\begin{proof}
If $\h$ is finite dimensional, one can let each unitary be the identity, and the result is the regular Montel theorem.
So we shall assume that $\h$ is infinite dimensional.
Let $\{e_i\}$ be an orthonormal basis for $\h$. For each fixed $k$ let
\[
\h^k = \spn \{e_1,e_2,\ldots,e_k\},
\]
\[
\m^k_i= \spn \{u^k(\lambda_1), u^k(\lambda_2), \ldots, u^k(\lambda_k)\},\qquad i=1,\ldots,k.
\]
For each $k$ choose a unitary $U^k\in \b(\h)$ satisfying
\[
U^k \m^k_i \subseteq \h^i,\qquad i=1,\ldots,k
\]
Observe that with this construction, for each fixed $i$,
\[
\{U^k u^k(\lambda_i)\}_{k = i}^\infty
 \]
 is a bounded sequence in $\h^i$, a finite dimensional Hilbert space.
Therefore, there exist $v_i \in \h$ and an increasing sequence of indices $\{k_l\}$ such that
\[
U^{k_l} u^{k_l}(\lambda_i) \to v_i \ \ \text { in }\  \h\ \ \text{ as }\  l \to \infty.
\]
Applying this fact successively with $i=1$, $i=2$, and so on, at each stage taking a subsequence of the previously selected subsequence, leads to a sequence $\{v_i\}$ in $\h$ and an  increasing sequence of indices $\{k_l\}$ such that
\[
U^{k_l} u^{k_l}(\lambda_i) \to v_i \ \ \text { in }\  \h\ \ \text{ as }\  l \to \infty.
\]
for all $i$. The lemma then follows if we let $V^l = U^{k_l}$.
\end{proof}

\noindent{\bf Proof of Theorem \ref{int.thm.10}.} Assume that $\Omega$ is an open set in $\c^d$, $\h$ is a Hilbert space, and $\{u^k\}$ is  a locally uniformly bounded sequence in $\hol_\h(\d)$. The theorem follows from the classical Montel Theorem (with $U^k = \id_\h$ for all $k$) if $\dim \h<\infty$. Therefore, we may assume that $\dim \h =\infty$.

 Fix a dense sequence $\{\lambda_i\}$ in $\Omega$. By Lemma \ref{com.lem.10}, there exists a subsequence $\{u^{k_l}\}$ and a sequence $\{V^l\}$
of unitary operators on $\h$ such that for each fixed $i$, $\{V^l u^{k_l}(\lambda_i)\}$ is a convergent sequence in $\h$. Furthermore, as $\{u^k\}$ is  locally uniformly bounded, so also, $\{V^l u^{k_l}\}$ is locally uniformly bounded. Therefore, Proposition \ref{int.prop.10} implies that $\{V^l  u^{k_l}\}$ is a convergent sequence in $\hol_\h(\Omega)$. The theorem then follows by choosing $\{U^k\}$ to be any sequence of unitaries in $\b(\h)$ such that $U^{k_l} = V^l$ for all $l$.
\ep
\section{Holomorphic Functions in Noncommuting Variables}
\label{secc}

If $\Omega$ is finitely open in $\bbm^d$, we may construct a \emph {finitely compact-open exhaustion of $\Omega$}, i.e., an increasing sequence of compact sets $\{K_i\}$ that satisfy
\[
K_1 \subset K_2^\circ\subset K_2\subset K_3^\circ \subset \ldots
\]
and  with $\Omega=\cup_i K_i$. For a set $E \subseteq \Omega$ and $f \in \hol(\Omega)$ we let
\[
\norm{f}_E =\sup_{\lambda \in E} \norm{f(\lambda)}
\]
and then in the usual way define a metric $d$ on $\hol(\Omega)$ with the formula
 \[
d(f,g) = \sum_{n=1}^\infty \frac{1}{2^n} \frac{\norm{f-g}_{K_n}}{1+\norm{f-g}_{K_n}}.
\]
It then follows that $f_k \to f$ in the metric space $(\Omega,d)$ if and only if for each finitely compact set $K$ in $\Omega$, $\{f_k\}$ converges uniformly to $f$ on $K$, i.e.,
\[
\lim_{k\to \infty} \norm{f-f_k}_K =0.
\]
Furthermore, $\hol(\Omega)$ is a complete metric space when endowed with this topology of uniform convergence on finitely compact subsets of $\Omega$.

It is a straightforward exercise to extend Montel's Theorem to the space $\hol_\h(\Omega)$ when $\dim \h$ is finite.
\begin{prop}\label{hol.thm.5}
If $\Omega$ is a finitely open set in $\bbm^d$, $\h$ is a Hilbert space with $\dim \h <\infty$, and $\{u^k\}$ is  a finitely locally uniformly bounded sequence in $\hol_\h(\Omega)$, then $\{u^{k}\}$ has a convergent subsequence.
\end{prop}
Also, with the setup we have just described, mere notational changes to the proof of  Proposition \ref{int.prop.10} yield a proof of the following proposition.
\begin{prop}\label{hol.prop.10}
Assume that $\Omega$ is a finitely open set in $\bbm^d$, $\{\lambda_i\}$ is a dense sequence in $\Omega$ with $\lambda_i \in \bbm_{n_i}^d$ for each $i$, and $\h$ is a Hilbert space. If $\{u^k\}$ is sequence in $\hol_\h(\Omega)$ that is finitely locally uniformly bounded on $\Omega$, and for each fixed $i$, $\{u^{k}(\lambda_i)\}$ is a convergent sequence in $\b(\c^{n_i},\c^{n_i}\otimes\h)$, then $\{ u^{k}\}$ is a convergent sequence in $\hol_\h(\Omega)$.
\end{prop}
Just as was the case for Proposition \ref{int.prop.10}, it is possible to relax the assumption in Proposition \ref{hol.prop.10} that $\{\lambda_i\}$ be a dense sequence in $\Omega$, to the assumption that $\{\lambda_i\}$ merely be a set of uniqueness for $\hol_\h(\Omega)$ (see Proposition \ref{uniq.prop.20}).

We now turn to an analog of Theorem \ref{int.thm.10} in the noncommutative setting.
\begin{lem}\label{hol.lem.10}
{\bf Wandering Isometry Lemma (noncommutative case).} Assume that $\Omega$ is a finitely open set in $\bbm^d$ and $\{\lambda_i\}$ is a sequence in $\Omega$ (where, for each $i$, $\lambda_i \in \bbm_{n_i}^d$). If $\h$ is an infinite dimensional Hilbert space and $\{u^k\}$ is sequence in $\hol_\h(\Omega)$ with the property that $\{u^k(\lambda_i)\}$ is bounded for each $i$, then there exists a subsequence $\{u^{k_l}\}$ and a sequence $\{V^l\}$
of unitary operators on $\h$ such that for each fixed $i$, $\{(\id_{n_i} \otimes V^l)\,  u^{k_l}(\lambda_i)\}$ is a convergent sequence in $\b(\c^{n_i},\c^{n_i} \otimes \h)$.
\end{lem}
\begin{proof}
Choose an increasing sequence $\{\h_i\}$ of subspaces of $\h$ with the property that
\[
\dim \h_1 = n_1^2\ \  \text{ and } \forall_{i\ge1}\ \
\dim (\h_{i+1} \ominus \h_i) = n_{i+1}^2,
\]
and for each $n$, let $\{e_1,\ldots,e_n\}$ denote the standard basis of $\c^n$.

Fix $k$. For each $i=1,\ldots,k$, as $u^k(\lambda_i):\c^{n_i} \to \c^{n_i} \otimes \h$, there exist $n_i^2$ vectors $x^{k,i}_{r,s}\in \h$, $r,s=1,\ldots, n_i$, such that
\be\label{hol.20}
u^k(\lambda_i)e_r = \sum_{s=1}^{n_i} e_s \otimes x^{k,i}_{r,s},\qquad r=1,\ldots,n_i.
\ee
For each $i=1,\ldots k$, define $\m^k_i$ by
\[
\m^k_i = \spn\set{x^{k,i}_{r,s}}{r,s =1,\ldots,n_i}.
\]
and then define a sequence of spaces $\{\n^k_i\}$ by setting $\n^k_1 =\m^k_1$ and
\[
\n^k_i = \big(\m^k_1 +\m^k_2 +\ldots +\m^k_i\big) \ominus
\big(\m^k_1 +\m^k_2 +\ldots +\m^k_{i-1}\big),
\]
for $i=2,\ldots, k$.
As for each $i=1,\ldots,k$, $\dim \m^k_i \le n_i^2$, so also, for $i=1,\ldots,k$, $\dim \n^k_i \le n_i^2$. Consequently, as the spaces $\{\n^k_i\}$ are also pairwise orthogonal, it follows that there exists a unitary $U^k \in \b(\h)$ such that
\[
U^k(\n^k_1) \subseteq \h_1 \text { and } U^k(\n^k_i) \subseteq \h_i \ominus \h_{i-1} \text{ for }i=2,\ldots, k.
\]
With this construction it follows using \eqref{hol.20} that
\be\label{hol.30}
(\id_{n_i} \otimes U^k)u^k(\lambda_i) (\c^{n_i}) \subseteq \c^{n_i} \otimes \h_i,\qquad i=1,\ldots,k.
\ee

Now observe that as $\eqref{hol.30}$ holds for each $k$,  for each fixed $i$,
\[
(\id_{n_i} \otimes U^k)u^k(\lambda_i) (\c^{n_i}) \subseteq \c^{n_i} \otimes \h_i,\qquad k=i,i+1,\ldots,
 \]
 i.e.,
 \[
 \big\{(\id_{n_i} \otimes U^k)u^k(\lambda_i)\big\}_{k=i}^\infty
 \]
 is a bounded sequence in $\b(\c^{n_i},\c^{n_i}\otimes \h_i)$, a finite dimensional Hilbert space.
Therefore, for each fixed $i$, there exist $L\in \h$ and an increasing sequence of indices $\{k_l\}$ such that
\[
U^{k_l} u^{k_l}(\lambda_i) \to L\ \ \text { in }\  \b(\c^{n_i},\c^{n_i}\otimes \h_i)\ \ \text{ as }\  l \to \infty.
\]
Applying this fact successively with $i=1$, $i=2$, and so on, at each stage taking a subsequence of the previously selected subsequence, leads to a sequence $\{L_i\}$ with $L_i \in \b(\c^{n_i},\c^{n_i}\otimes \h_i)$ for each $i$ and an  increasing sequence of indices $\{k_l\}$ such that
\[
\forall_i\ \ U^{k_l} u^{k_l}(\lambda_i) \to L_i \ \ \text { in }\  \b(\c^{n_i},\c^{n_i}\otimes \h_i)\ \ \text{ as }\  l \to \infty.
\]
The lemma then follows if we let $V^l = U^{k_l}$.
\end{proof}
Lemma \ref{hol.lem.10} suggests the following notation. Let $\Omega$ be a finitely open set in $\bbm^d$ and let $\h$ be a Hilbert space. If $U$ is a unitary acting on $\h$, and $f \in \hol_\h(\Omega)$ then we may define $U*f \in \hol_\h(\Omega)$ by the formula
\[
\forall_n\ \  (U*f)|\,\Omega_n = (\id_n \otimes U) f|\, \Omega_n.
\]
With this notation we may formulate a noncommutative analog of Theorem \ref{int.thm.10} in the noncommutative setting.
\begin{thm}\label{hol.thm.10}
If $\Omega$ is a finitely open set in $\bbm^d$, $\h$ is a Hilbert space, and $\{u^k\}$ is  a finitely locally uniformly bounded sequence in $\hol_\h(\Omega)$, then there exists a sequence $\{U^k\}$
of unitary operators on $\h$ such that $\{U^k * u^{k}\}$ has a convergent subsequence.
\end{thm}
\begin{proof}
Assume that $\Omega$ is an open set in $\bbm^d$, $\h$ is a Hilbert space, and $\{u^k\}$ is  a finitely locally uniformly bounded sequence in $\hol_\h(\Omega)$. If $\dim \h < \infty$, then the theorem follows from Proposition \ref{hol.thm.5} if we choose $U^k =\id_\h$ for all $k$. Therefore, we assume that $\dim \h = \infty$.

 Fix a dense sequence $\{\lambda_i\}$ in $\Omega$. By Lemma \ref{hol.lem.10}, there exists a subsequence $\{u^{k_l}\}$ and a sequence $\{V^l\}$
of unitary operators on $\h$ such that for each fixed $i$, $\{V^l u^{k_l}(\lambda_i)\}$ is a convergent sequence in $\b(\c^{n_i},\c^{n_i} \otimes \h)$. Furthermore, as $\{u^k\}$ is  locally uniformly bounded, so also, $\{V^l u^{k_l}\}$ is locally uniformly bounded. Therefore, Proposition \ref{int.prop.10} implies that $\{V^l  u^{k_l}\}$ is a convergent sequence in $\hol_\h(\Omega)$. The theorem then follows by choosing $\{U^k\}$ to be any sequence of unitaries in $\b(\h)$ such that $U^{k_l} = V^l$ for all $l$.
\end{proof}
\section{Locally bounded nc Functions}
\label{secd}

Properties of $\tau$-holomorphic functions can be very sensitive to the choice of nc-topology $\tau$. For example, if $\tau$ is the \emph{ fat topology }  studied in \cite{amif16}, then $\tau$-holomorphic functions satisfy a version of the Implicit Function Theorem. On the other hand if $\tau$ is the \emph{free topology}, studied in \cite{amfree}, then $\tau$-holomorphic functions satisfy the Oka-Weil Approximation Theorem. Remarkably, neither of these theorems holds for the other topology.

Also notice that if $f$ is $\tau$-holomorphic  in the sense of Definition \ref{nc.def.20}, then necessarily $\Omega$, the domain of $f$, is an open set in the $\tau$ topology:
for each $\lambda\in \Omega$ there exists $B_\lambda \subseteq \Omega$ such that $\lambda \in B_\lambda \in \tau$; hence, $
\Omega = \bigcup_{\lambda} B_\lambda \in \tau$.

There are no such subtleties between the nc-topologies when it comes to understanding the implications of local boundedness.
\begin{defin}\label{nc.def.30}
Assume that $\tau$ is an nc-topology and $\Omega \in \tau$. If $\{u^k\}$ is a sequence in $\hol_\h^\tau(\Omega)$,
we say that $\{u^k\}$ is \emph{$\tau$-locally uniformly bounded on $\Omega$} if for each $\lambda\in \Omega$, there exists a $\tau$-open  $B\subseteq \Omega$ such that, $\lambda \in B $ and
\[
\sup_k \norm{u^k}_B <\infty.
\]
\end{defin}
\begin{lem}\label{nc.lem.10}
Assume that $\tau$ is an nc-topology and $\Omega \in \tau$. Let $u \in \hol(\Omega)$ and $\{u^k\}$ be a sequence in $\hol_\h^\tau(\Omega)$. If $\{u^k\}$ is $\tau$-locally uniformly bounded on $\Omega$ and $u^k \to u$ in $\hol_\h(\Omega)$, then $u\in \hol^\tau_\h(\Omega)$.
\end{lem}
\begin{proof} Under the assumptions of the lemma, we need to prove the following two assertions.
\be\label{nc.10}
u \text{ is an nc-function on }\Omega.
\ee
\be\label{nc.20}
u \text{ is $\tau$-locally bounded on }\Omega.
\ee

To prove \eqref{nc.10}, note first that as $u\in \hol_\h(\Omega)$, Condition (i) in Definition \ref{nc.def.10} holds. To verify Condition (ii), assume that $\lambda,\mu,\lambda \oplus \mu \in \Omega$. Then, as $u^k \to u$ in $\hol_\h(\Omega)$ and $u^k \in \hol^\tau_\h$ for all $k$,
\begin{align*}
u(\lambda \oplus \mu)& =\lim_{k \to \infty} u^k(\lambda \oplus \mu)\\
&=\lim_{k \to \infty} \big(u^k(\lambda) \oplus u^k(\mu)\big)\\
&=\lim_{k \to \infty} u^k(\lambda) \oplus \lim_{k\to \infty} u^k(\mu)\\
&=u(\lambda) \oplus u(\mu).
\end{align*}
Finally, note that if $n\ge 1$, $S \in \bbm_n$ is invertible, and both $\lambda$ and $S\lambda S^{-1}$ are in $\Omega_n$, then
\begin{align*}
u(S\lambda S^{-1})&=\lim_{k\to \infty}u^k(S\lambda S^{-1})\\
&=\lim_{k\to \infty}(S \otimes \id_\h)\, u^k(\lambda)\,  S^{-1}\\
&=(S \otimes \id_\h)\, u(\lambda)\,  S^{-1},
\end{align*}
which proves Condition (iii).

To prove \eqref{nc.20}, fix $\lambda \in \Omega$. As $\{u^k\}$ is $\tau$-locally uniformly bounded on $\Omega$, Definition \ref{nc.def.30} implies that there exist $B\subseteq \Omega$ and a constant $\rho$ such that $\lambda \in B \in \tau$ and
\[
\sup_k \norm{u^k}_B \le \rho.
\]
Fix $\mu \in B$. As we assume that $u^k \to u$ in $\hol_\h(\Omega)$, it follows that
\[
\norm{u(\mu)} =\lim_{k\to \infty}\norm{u^k(\mu)} \le \rho.
\]
But then,
\[
\norm{u}_B \le \rho.
\]
As $B \in \tau$, this proves that $u$ is $\tau$-locally bounded on $\Omega$.
\end{proof}
Definition \ref{nc.def.30} and Lemma \ref{nc.lem.10}, allow one to easily deduce
Theorem~\ref{nc.thm.101}
 as a corollary of  Theorem \ref{hol.thm.10}.

{\sc Proof of Theorem~\ref{nc.thm.101}:}
As we assume that $\{u^k\}$ is  a $\tau$-locally uniformly bounded sequence in $\hol_\h^\tau(\Omega)$, $\{u^k\}$ is finitely locally uniformly bounded in $\hol_\h(\Omega)$. Therefore, Theorem \ref{hol.thm.10} implies that there exists a sequence $\{U^k\}$
of unitary operators on $\h$ such that $\{U^k * u^k \}$ has a subsequence that converges in $\hol_\h(\Omega)$. Consequently, we may choose $u\in \hol_\h(\Omega)$  and an increasing sequence of indices $\{k_l\}$ such that
$U^{k_l} * u^{k_l} \to u$ in $\hol(\Omega)$. The proof 
is completed by observing that Lemma \ref{nc.lem.10} implies that $u \in \hol_\h^\tau(\Omega)$.
\ep
Let us emphasize that Theorem~\ref{nc.thm.101} asserts that $U^{k_l} * u^{k_l}$ converges to $u$, which 
is in  $\hol_\h^\tau(\Omega)$, uniformly
on sets that are compact in the finite topology; it does not say that 
it converges uniformly on  compact sets in the $\tau$ topology.

Note that the proofs of Lemma~\ref{nc.lem.10} and Theorem~\ref{nc.thm.101}
work identically if $u^k$ are just assumed to be in $\holnc$; so we get
\begin{thm}
\label{thmnc2}
Let $\O$ be 
 a finitely open set in $\bbm^d$, $\h$ is a Hilbert space, and $\{u^k\}$ is  a finitely locally uniformly bounded sequence in $\holnc$, then there exists a sequence $\{U^k\}$
of unitary operators on $\h$ such that $\{U^k * u^{k}\}$ has a  subsequence that converges finitely locally uniformly 
to an element of $\holnc$.
\end{thm}

\section{Some Applications}
\label{sece}

A useful construct in the study of $\tau$-holomorphic functions is the \emph{duality construction}. If $\Omega$ is a finitely open set it is natural to consider the algebraic tensor product $\hol(\Omega)^* \otimes \hol(\Omega)$. This space can concretely be realized as the set of functions $A$ defined on
 \[
 \Omega \boxtimes \Omega =\bigcup_{n=1}^\infty \big(\Omega\cap \mathbb{M}_n^d\big) \times \big(\Omega\cap \mathbb{M}_n^d\big),
 \]
 and such that there exist a finite dimensional Hilbert space $\h$ and $u,v \in \hol_\h(\Omega)$ such that
 \[
 A(\lambda,\mu) = v(\mu)^*u(\lambda),\qquad (\lambda,\mu) \in  \Omega \boxtimes \Omega.
 \]
 As the functions in $\hol(\Omega)^* \otimes \hol(\Omega)$ are holomorphic in $\lambda$ for each fixed $\mu$ and anti-holomorphic in $\mu$ for each fixed $\lambda$, we may complete  $\hol(\Omega)^* \otimes \hol(\Omega)$ in the topology of uniform convergence on finitely compact subsets of  $\Omega \boxtimes \Omega$ to obtain the space of \emph{hereditary holomorphic functions on $\Omega$}, $\her(\Omega)$. Inside $\her(\Omega)$, we may define a cone $\calp$ by
 \[
 \calp=\set{u(\mu)^*u(\lambda)}{u \in \hol_\h(\Omega) \text{ for some Hilbert space } \h}.
 \]
 \begin{thm}\label{app.thm.10}
 $\calp$ is closed in $\her(\Omega)$.
 \end{thm}
 \begin{proof}
 Assume that $\{v^k\}$ is a sequence with $v^k \in \hol_{\h_k}(\Omega)$ for each $k$ and with $v^k(\mu)^*v^k(\lambda) \to A$ in $\her(\Omega)$. We may assume that $\h_k$ is separable for each $k$. Fix a separable infinite dimensional Hilbert space $\h$ and for each each $k$ choose an isometry $V^k:\h_k \to \h$. If for each $k$ we let $u^k = V^k *v^k$, then $\{u^k\}$ is a sequence in $\hol_\h(\Omega)$ and  $u^k(\mu)^*u^k(\lambda) \to A$ in $\her(\Omega)$.

 Now, as  $u^k(\mu)^*u^k(\lambda) \to A$ in $\her(\Omega)$ it follows that $\{u^k\}$ is a finitely locally uniformly bounded sequence in $\hol_\h(\Omega)$. Hence, by Theorem \ref{hol.thm.10}, there exists a sequence $U^k$ of unitary operators on $\h$ such that $\{U^k*u^k\}$ has a convergent subsequence, i.e., there exists $u\in \hol_\h(\Omega)$ and an increasing sequence of indices $\{k_l\}$ such that $U^{k_l}*u^{k_l} \to u$. But then, for each $(\lambda,\mu) \in \Omega \boxtimes \Omega$,
 \begin{align*}
 A(\lambda,\mu)&=\lim_{k\to \infty} u^k(\mu)^*u^k(\lambda)\\
 &=\lim_{l\to \infty} u^{k_l}(\mu)^*u^{k_l}(\lambda)\\
  &=\lim_{l\to \infty}
  (U^{k_l}*u^{k_l})(\mu)^*(U^{k_l}*u^{k_l})(\lambda)\\
  &=u(\mu)^*u(\lambda),
 \end{align*}
 i.e., $A \in \calp$.
 \end{proof}

 We also may may use wandering Montel Theorems to study sums of $\tau$-holomorphic dyads. We let $\her^\tau(\Omega)$ denote the closure of
 \[
\set{v(\mu)^*u(\lambda)}{u,v \in \hol_\h^\tau(\Omega) \text{ for some finite dimensional Hilbert space } \h}
 \]
 inside $\her(\Omega)$ and define $\calp^\tau$ in $\her^\tau(\Omega)$ by
 \[
  \calp^\tau=\set{u(\mu)^*u(\lambda)}{u \in \hol_\h^\tau(\Omega) \text{ for some Hilbert space } \h}.
 \]
  \begin{thm}\label{app.thm.20}
  Let $\tau$ be an nc-topology, and $\Omega \in \tau$. Then
 $\calp^\tau$ is closed in $\her^\tau(\Omega)$.
 \end{thm}
 \begin{proof}
Assume that $u^k(\mu)^*u^k(\lambda) \to A$ in $\her(\Omega)$, where, as in the proof of Theorem \ref{app.thm.10}, we may assume that $u^k \in \hol_\h^\tau(\Omega)$ for each $k$. 
By Theorem \ref{nc.thm.101}, there exist $u\in \hol_\h^\tau(\Omega)$, a sequence $U^k$ of unitary operators on $\h$, and and an increasing sequence of indices $\{k_l\}$ such that $U^{k_l}*u^{k_l} \to u$. But then as in the proof of Theorem \ref{app.thm.10}, $A(\lambda,\mu) =u(\mu)^*u(\lambda)$ for all $(\lambda,\mu) \in \Omega \boxtimes \Omega$, i.e., $A \in \calp^\tau$.
 \end{proof}
 
 Finally, we shall prove that the model cone is closed; this is the key ingredient in the proof of the realization formula for free holomorphic functions \cite{amfree, bmv16b}.
 Let $\de$ be a $J$-by-$L$ matrix whose entries are free polynomials in $d$ variables.
 We define $\bd$ to be the polynomial polyhedron
 \[
 \bd \ := \ \{ x \in \md : \| \de(x) \| < 1 \}.
 \]
 The free topology is the nc-topology generated by the sets $\bd$, as $\de$ ranges over all matrices of 
 polynomials. 
 The model cone $\calc$ is the set of hereditary functions on $\bd$ of the form
 \be
 \label{defe1}
 \calc \  := \ 
 \{ 
 \tensor{\id_{\c^J}}{u(\mu)^*}  \left(  \id - 
 \tensor{\de(\mu)^* \de (\lambda ) }{\id_\h} \right)
  \tensor{\id_{\c^J}}{u(\lambda)}
  \ : \ u \in \hol_\h(\bd) \ {\rm and\ }u\ {\rm is\ nc},\  \ {\rm for\ some\ Hilbert\ space}\ \h 
 \}.
 \ee
 We write the tensors vertically just to enhance readability.
 \begin{thm}\label{app.thm.30}
 The model cone 
 $\calc$, defined in \eqref{defe1} is closed in $\her(\bd)$.
 \end{thm}
 \begin{proof}
 Suppose $u^k$ is a sequence of nc functions  in $ \hol_\h(\bd)$ (we may assume the space $\h$ is the same for each $u^k$, as in the proof of Thm.~\ref{app.thm.10}),
 so that
 \be
 \label{eqe5}
  \tensor{\id_{\c^J}}{u^k(\mu)^*}  \left(  \id - 
 \tensor{\de(\mu)^* \de (\lambda ) }{\id_\h} \right)
  \tensor{\id_{\c^J}}{u^k(\lambda)}
 \ee
 converges in $\her(\bd)$ to $A(\lambda,\mu)$.
 On any finitely compact set, $\| \de(x) \|$ will be bounded by a constant that is strictly less than one.
 Since \eqref{eqe5} converges uniformly on finitely compact subsets of
 $\bd \boxtimes \bd$, this means that $u^k$ is a finitely  locally uniformly bounded sequence.
 Therefore by Thm.~\ref{hol.thm.10}, there exist unitaries $U^k$ such that $U^k * u^k$ has a convergent
 subsequence, which converges to some nc function $u \in \hol_\h(\bd)$.
 Then
 \[
 A(\lambda, \mu) \=
 \tensor{\id_{\c^J}}{u(\mu)^*}  \left(  \id - 
 \tensor{\de(\mu)^* \de (\lambda ) }{\id_\h} \right)
  \tensor{\id_{\c^J}}{u(\lambda)},
  \]
  as desired.

 \end{proof}
 
\section{Sets of Uniqueness}
\label{secf}

In this section we shall show that the assumption in Propositions \ref{int.prop.10} and \ref{hol.prop.10}, that $\{\lambda_i\}$ be a dense sequence in $\Omega$, can be relaxed to the assumption that $\{\lambda_i\}$ be a set of uniqueness for $\hol(\Omega)$. We remark that it is an elementary fact that if $\h$ is a Hilbert space, then $\{\lambda_i\}$ is a set of uniqueness for $\hol_\h(\Omega)$ if and only if $\{\lambda_i\}$ is a set of uniqueness for $\hol(\Omega)$.

The following proposition is essentially the same as the Arendt-Nikolski Theorem~\ref{thman}, 
so we shall omit the proof.

\begin{prop}\label{uniq.prop.10}
Assume that $\Omega$ is an open set in $\c^d$,  $\{\lambda_i\}$ is a sequence in $\Omega$ that is a set of uniqueness for $\hol_\h(\Omega)$
\footnote
 {i.e., if $f\in \hol(\Omega)$ and $f(\lambda_i) =0$ for all $i$, then $f(\lambda)=0$ for all $\lambda \in \Omega$.},
 and $\h$ is a Hilbert space.
 If $\{u^k\}$ is sequence in $\hol_\h(\Omega)$ that is locally uniformly bounded on $\Omega$, and for each fixed $i$ $\{u^k(\lambda_i)\}$ is a convergent sequence in $\h$, then $\{u^k\}$ converges in $\hol_\h(\Omega)$.
\end{prop}
Here is the graded version.

\begin{prop}\label{uniq.prop.20}
Assume that $\Omega$ is a finitely open set in $\c^d$,  $\{\lambda_i\}$ is a sequence in $\Omega$ (with $\lambda_i \in \bbm_{n_i}^d$ for each $i$) that is a set of uniqueness for $\hol_\h(\Omega)$,
 and $\h$ is a Hilbert space.
 If $\{u^k\}$ is sequence in $\hol_\h(\Omega)$ that is finitely locally uniformly bounded on $\Omega$, and for each fixed $i$ $\{u^k(\lambda_i)\}$ is a convergent sequence in $\h$, then $\{u^k\}$ converges in $\hol_\h(\Omega)$.
\end{prop}
\begin{proof}
The theorem will follow if we can show that $\{u^k|\Omega_n\}$ is a convergent sequence for each $n$. Accordingly, fix $n$ and adopt the notation $H_n$ for the holomorphic $\b(\c^n,\c^n \otimes \h)$-valued functions defined on $\Omega_n$. Thus, $\{u^k|\Omega_n\}$ is a locally uniformly bounded sequence in $H_n$. Furthermore,  if $\{\eta_j\}$ is an enumeration of $\set{\lambda_i}{i \ge 1}\cap \Omega_n$, as $\{\lambda_i\}$ is a set of uniqueness for $\hol(\Omega)$, $\{\eta_i\}$ is a set of uniqueness for both $\hol(\Omega_n)$ and $H_n$. Finally, let $u^k(\eta_j) \to u_j$ as $k\to \infty$ for each $j$.

For fixed $\alpha \in \c^n$ and $\beta \in \c^n$, define $f^k_{\alpha,\beta} \in \hol(\Omega_n)$ by
\be\label{uniq.140}
f^k_{\alpha,\beta}(\lambda) = \ip{u^k(\lambda) \alpha}{\beta}_{\c^n \otimes \h},\qquad \lambda \in \Omega.
\ee
Noting that,
\be\label{uniq.150}
|f_{\alpha,\beta}^k(\lambda)|=|\ip{u^k(\lambda) \alpha}{\beta}|
\le \norm{u^k(\lambda)}\norm{\alpha}\norm{\beta},
\ee
it follows that  $\{f^k_{\alpha,\beta}\}$ is locally uniformly bounded on $\Omega_n$. Therefore by Montel's Theorem, $\{f_{\alpha,\beta}^k\}$ has compact closure in $\hol(\Omega_n)$.

We claim that $\{f_{\alpha,\beta}^k\}$ has a unique cluster point. For assume that $\{f^{k_r}_{\alpha,\beta}\}$ and $\{f^{k_s}_{\alpha,\beta}\}$ are subsequences of $\{f^{k}_{\alpha,\beta}\}$ with $\{f^{k_r}_{\alpha,\beta}\} \to f$ and $\{f^{k_s}_{\alpha,\beta}\} \to g$.  Then, as we assume for each $j$, $u^k(\eta_j)\to u_j$ as $k \to \infty$,
\begin{align*}
f(\eta_i)&=\lim_{r\to \infty} f^{k_r}_{\alpha,\beta}(\eta_i)\\
&=\lim_{r\to \infty}\ip{u^{k_r}(\eta_i)\alpha}{\beta}\\
&=\ip{u_i \alpha}{\beta}\\
&=\lim_{s\to \infty}\ip{u^{k_s}(\eta_i)\alpha}{\beta}\\
&=\lim_{s\to \infty}f^{k_s}_{\alpha,\beta}(\eta_i)\\
&=g(\eta_i).
\end{align*}
Hence, as $\{\eta_i\}$ is a set of uniqueness, $f=g$. Since $\{f_{\alpha,\beta}^k\}$ has a unique cluster point, we have shown that for each $\alpha \in \c^n$ and $\beta\in \c^n \otimes \h$, there exists $f_{\alpha,\beta} \in \hol(\Omega_n)$ such that
\be\label{uniq.160}
f_{\alpha,\beta}^k \to f_{\alpha,\beta} \text{ in } \hol(\Omega_n) \text{ as } k\to \infty.
\ee

Now fix $\lambda \in \Omega_n$ and define $L_\lambda$ by
 \be\label{uniq.170}
 L_\lambda(\alpha,\beta)=f_{\alpha,\beta} (\lambda),\qquad \alpha\in \c^n, \beta \in \c^n\otimes \h.
 \ee
 Observe that \eqref{uniq.140} and \eqref{uniq.160} imply that $L_\lambda$ is a sesqui-linear functional on $\c^n \times (\c^n \otimes \h)$. Furthermore, \eqref{uniq.150} and \eqref{uniq.160} imply that $L_\lambda$ is bounded. Therefore, by the Riesz Representation Theorem, there exists $u(\lambda) \in \b(\c^n,\c^n \otimes\h)$ such that
\[
\forall_{\alpha \in \c^n}\ \
\forall_{\beta \in \c^n \otimes \h}\ \ L_\lambda(\alpha,\beta) = \ip{u(\lambda)\alpha}{\beta},
\]
or equivalently,
\[
\forall_{\alpha\in \c^n}\ \ \forall_{\beta \in \c^n \otimes \h}\ \ \ip{u(\lambda)\alpha}{\beta}=f_{\alpha,\beta}(\lambda).
\]

The function $u$ constructed in the previous paragraph has the following properties: it is holomorphic,
\be\label{uniq.180}
\forall_{\lambda\in \Omega_n}\ \  u^k(\lambda) \to u(\lambda)\ \
 \text{ weakly in }\b(\c^n,\c^n\otimes\h) \ \ \text{ as }\ \  k\to \infty,
 \ee
 and
 \be\label{uniq.190}
\forall_{j}\ \   u^k(\eta_j) \to u(\eta_j)\ \text{ in norm in $\b(\c^n,\c^n\otimes \h)$\ \   as }\ \ k\to \infty.
 \ee

\begin{claim}\label{uniq.claim.20}
\[
u^k(\mu)^*u^k(\lambda) \to u(\mu)^*u(\lambda)\ \
\text{ in }\ \ \her(\Omega_n)\ \
\text{ as }\ \ k\to \infty.
\]
\end{claim}
To prove this claim, first note that as we are assuming $\{u^k\}$ is a locally uniformly bounded sequence in $\hol_\h(\Omega_n)$, $\{u^k(\mu)^*u^k(\lambda) \}$ is a locally uniformly bounded sequence in $\her(\Omega_n)$. Therefore, the claim follows from Montel's Theorem if we can show that
\be\label{uniq.200}
A(\lambda,\mu)=u(\mu)^*u(\lambda)
\ee
whenever $\{k_r\}$ is a sequence of indices such that
\be\label{uniq.210}
u^{k_r}(\mu)^*u^{k_r}(\lambda) \to A(\lambda,\mu)\ \
\text{ in }\ \ \her(\Omega_n)\ \
\text{ as }\ \ r\to \infty.
\ee
But if \eqref{uniq.210} holds, then \eqref{uniq.190} implies that for each independently chosen $i$ and $j$,
\[
A(\eta_j,\eta_i)=\lim_{r \to \infty}u^{k_r}(\eta_i)^*u^{k_r}(\eta_j)
=u^{k_r}(\eta_i)^*u^{k_r}(\eta_j).
\]
Since both sides of \eqref{uniq.200} are holomorphic in $\lambda$ and anti-holomorphic in $\mu$, and $\{\eta_i\}$ is a set of uniqueness, it follows that \eqref{uniq.200} holds for all $\lambda,\mu \in \Omega$. This completes the proof of Claim \ref{uniq.claim.20}.

Finally, fix $\lambda \in \Omega$. By \eqref{uniq.180}, $\{u^k(\lambda)\}$ converges weakly in $\b(\c^n,\c^n\otimes\h)$ to $u(\lambda)$ and by Claim \ref{uniq.claim.20},  $u^k(\lambda)^*u^k(\lambda) \to u(\lambda)^*u(\lambda)$. Therefore, $u^k(\lambda) \to u(\lambda)$ in norm in $\b(\c^n,\c^n\otimes \h)$. Since this holds for all $\lambda \in \Omega$, the proof of Proposition \ref{uniq.prop.20} may be completed by an application of Proposition \ref{hol.prop.10}.
\end{proof}
 \bibliography{../../references}

\begin{thebibliography}{10}

\bibitem{amfreeII}
J.~Agler and J.E. M\raise.45ex\hbox{c}Carthy.
\newblock Global holomorphic functions in several non-commuting variables {II}.
\newblock To appear.

\bibitem{amfree}
J.~Agler and J.E. M\raise.45ex\hbox{c}Carthy.
\newblock Global holomorphic functions in several non-commuting variables.
\newblock {\em Canad. J. Math.}, 67(2):241--285, 2015.

\bibitem{amif16}
J.~Agler and J.E. M\raise.45ex\hbox{c}Carthy.
\newblock The implicit function theorem and free algebraic sets.
\newblock {\em Trans. Amer. Math. Soc.}, 368(5):3157--3175, 2016.

\bibitem{akv06}
D.~Alpay and D.~S. Kalyuzhnyi-Verbovetzkii.
\newblock Matrix-{$J$}-unitary non-commutative rational formal power series.
\newblock In {\em The state space method generalizations and applications},
  volume 161 of {\em Oper. Theory Adv. Appl.}, pages 49--113. Birkh\"auser,
  Basel, 2006.

\bibitem{an00}
W.~Arendt and N.~Nikolski.
\newblock Vector-valued holomorphic functions revisited.
\newblock {\em Math. Z.}, 234(4):777--805, 2000.

\bibitem{bal15}
Sriram Balasubramanian.
\newblock Toeplitz corona and the {D}ouglas property for free functions.
\newblock {\em J. Math. Anal. Appl.}, 428(1):1--11, 2015.

\bibitem{bmv16b}
J.A. Ball, G.~Marx, and V.~Vinnikov.
\newblock Interpolation and transfer function realization for the
  non-commutative {Schur-Agler} class.
\newblock To appear.

\bibitem{bgm06}
Joseph~A. Ball, Gilbert Groenewald, and Tanit Malakorn.
\newblock Conservative structured noncommutative multidimensional linear
  systems.
\newblock In {\em The state space method generalizations and applications},
  volume 161 of {\em Oper. Theory Adv. Appl.}, pages 179--223. Birkh\"auser,
  Basel, 2006.

\bibitem{bmv16a}
Joseph~A. Ball, Gregory Marx, and Victor Vinnikov.
\newblock Noncommutative reproducing kernel {H}ilbert spaces.
\newblock {\em J. Funct. Anal.}, 271(7):1844--1920, 2016.

\bibitem{hkm11b}
J.~William Helton, Igor Klep, and Scott McCullough.
\newblock Proper analytic free maps.
\newblock {\em J. Funct. Anal.}, 260(5):1476--1490, 2011.

\bibitem{hm12}
J.~William Helton and Scott McCullough.
\newblock Every convex free basic semi-algebraic set has an {LMI}
  representation.
\newblock {\em Ann. of Math. (2)}, 176(2):979--1013, 2012.

\bibitem{hptv16}
J.~William Helton, J.~E. Pascoe, Ryan Tully-Doyle, and Victor Vinnikov.
\newblock Convex entire noncommutative functions are polynomials of degree two
  or less.
\newblock {\em Integral Equations Operator Theory}, 86(2):151--163, 2016.

\bibitem{kvv14}
Dmitry~S. Kaliuzhnyi-Verbovetskyi and Victor Vinnikov.
\newblock {\em Foundations of free non-commutative function theory}.
\newblock AMS, Providence, 2014.

\bibitem{kk03}
Kang-Tae Kim and Steven~G. Krantz.
\newblock Normal families of holomorphic functions and mappings on a {B}anach
  space.
\newblock {\em Expo. Math.}, 21(3):193--218, 2003.

\bibitem{pas14}
J.~E. Pascoe.
\newblock The inverse function theorem and the {J}acobian conjecture for free
  analysis.
\newblock {\em Math. Z.}, 278(3-4):987--994, 2014.

\bibitem{ptd13}
J.E. Pascoe and R.~Tully-Doyle.
\newblock Free {Pick} functions: representations, asymptotic behavior and
  matrix monotonicity in several noncommuting variables.
\newblock arXiv:1309.1791.

\bibitem{po06}
Gelu Popescu.
\newblock Free holomorphic functions on the unit ball of {$B({\cal H})^n$}.
\newblock {\em J. Funct. Anal.}, 241(1):268--333, 2006.

\end{thebibliography}
 \end{document}